\documentclass{amsart}
\usepackage{amsfonts, amsbsy, amsmath, amssymb}

\makeatletter
\@namedef{subjclassname@2010}{%
  \textup{2020} Mathematics Subject Classification}
\makeatother

\usepackage{scalerel}
\DeclareMathOperator*{\bboxplus}{\scalerel*{\boxplus}{1\over 1}}

\ifx\fiverm\undefined 
  \newfont\fiverm{cmr5} 
\fi 

\input prepictex.tex
\input pictex.tex
\input postpictex.tex

\newtheorem{thm}{Theorem}[section]
\newtheorem{lem}[thm]{Lemma}

\newtheorem{thm-con}[thm]{Theorem-Conjecture}
\numberwithin{equation}{section}

\theoremstyle{definition}

\allowdisplaybreaks

\newcommand{\f}{\Bbb F}

\begin{document}

\title[Number of Affine Equivalence Classes of Boolean Functions]{On the Number of Affine Equivalence Classes of Boolean Functions and $q$-ary Functions}

\author[Xiang-dong Hou]{Xiang-dong Hou}
\address{Department of Mathematics and Statistics,
University of South Florida, Tampa, FL 33620}
\email{xhou@usf.edu}



\begin{abstract}
Let $R_q(r,n)$ be the $r$th order $q$-ary Reed-Muller code of length $q^n$, which is the set of functions from $\f_q^n$ to $\f_q$ represented by polynomials of degree $\le r$ in $\f_q[X_1,\dots,X_n]$. The affine linear group $\text{AGL}(n,\Bbb F_q)$ acts naturally on $R_q(r,n)$. We derive two formulas concerning the number of orbits of this action: (i) an explicit formula for the number of AGL orbits of $R_q(n(q-1),n)$, and (ii) an asymptotic formula for the number of AGL orbits of $R_2(n,n)/R_2(1,n)$. The number of AGL orbits of $R_2(n,n)$ has been numerically computed by several authors for $n\le 31$; the binary case of result (i) is a theoretic solution to the question. Result (ii) answers a question by MacWilliams and Sloane. 
\end{abstract}
\maketitle

\section{Introduction}

Let $\f_q$ be the finite field with $q$ elements and let $\mathcal F(\f_q^n,\f_q)$ denote the set of all functions from $\f_q^n$ to $\f_q$. Every $g\in\mathcal F(\f_q^n,\f_q)$ is (uniquely) represented by a polynomial $f\in\f_q[X_1,\dots,X_n]$ with $\deg_{X_i}f<q$ for all $1\le i\le n$; we define $\deg g=\deg f$. We shall not distinguish a function from $\f_q^n$ to $\f_q$ and a polynomial in $\f_q[X_1,\dots,X_n]$ that represents it. For $-1\le r\le n(q-1)$, the $r$th order Reed-Muller code of length $q^n$ is
\begin{equation}\label{1.1}
R_q(r,n)=\{f\in\mathcal F(\f_q^n:\f_q):\deg f\le r\}.
\end{equation}
Note that $\mathcal F(\f_q^n,\f_q)=R_q(n(q-1),n)$.
Let 
\begin{equation}\label{1.2}
\text{AGL}(n,\f_q)=\Bigl\{\left[\begin{matrix}A&0\cr a&1\end{matrix}\right]: A\in\text{GL}(n,\f_q),\ a\in\f_q^n\Bigr\}
\end{equation}
be the affine linear group of degree $n$ over $\f_q$. The set $\mathcal F(\f_q^n,\f_q)$ is an $\f_q$-algebra on which $\text{AGL}(n,\f_q)$ acts as automorphisms: For $\alpha=\left[\begin{smallmatrix}A&0\cr a&1\end{smallmatrix}\right]\in \text{AGL}(n,\f_q)$ and $f(X_1,\dots,X_n)\in\mathcal F(\f_q^n,\f_q)$,
\[
\alpha(f)=f((X_1,\dots,X_n)A+a).
\]
Consequently, $\text{AGL}(n,\f_q)$ acts on $R_q(r,n)$ and on $R_q(r,n)/R_q(s,n)$ for $-1\le s\le r\le n(q-1)$.

When $q=2$, we write $R_2(r,n)=R(r,n)$. In this case, $\mathcal F(\f_2^n,\f_2)=R(n,n)$ is the set of all Boolean functions in $n$ variables. When two Boolean functions are said to be {\em equivalent}, it is meant, depending on different authors, that $f$ and $g$ are in the same AGL orbit of $\mathcal F(\f_2^n,\f_2)$ \cite{Budaghyan-Carlet-Pott-IEEE-IT-2006}, or $f+R(1,n)$ and $g+R(1,n)$ are in the same AGL orbit of $R(n,n)/R(1,n)$ \cite{MacWilliams-Sloane-1977}. (The affine equivalence in the latter sense is referred to as {\em extended affine equivalence} in \cite{Budaghyan-Carlet-Pott-IEEE-IT-2006}.) Most coding theoretic and cryptographic properties of Boolean functions are preserved under affine equivalence. Let $\frak N_{q,n}$ and $\frak M_n$ denote the number of AGL orbits of $\mathcal F(\f_q^n,\f_q)$ and the number of AGL orbits of $R(n,n)/R(1,n)$, respectively. The number $\frak N_{q,n}$ has been computed by several authors for $q=2$ and $n$ up to 31 \cite{Harrison-JSIAM-1964, Hou-JA-1995, Nechiporuk-DKNSSSR-1958, Zhang-Yang-Hung-Zhang-IEEE-C-2016, Zivkovic-Caric-IEEE-IT-2021}. We will derive an explicit formula for $\frak N_{q,n}$, hence providing a theoretic solution to the question. Our approach differs from those in some previous works in that we do not use the cycle index, the generating function in P\'olya's counting; rather, we find direct application of Burnside's lemma more suitable for this particular question. The number $\frak M_n$ has also been studied by a number of authors \cite{Berlekamp-Welch-IEEE-IT-1972, Hou-JA-1995, MacWilliams-Sloane-1977, Maiorana-MC-1991, Strazdins-AAM-1997, Zeng-Yang-arXiv1912.11189}. (In fact, the number of AGL orbits of $R(r,n)/R(s,n)$ has been computed recently for all $-1\le s<r\le n\le 10$ \cite{Zeng-Yang-arXiv1912.11189}.) In \cite{Maiorana-MC-1991}, Maiorana not only computed $\frak M_6$, but also classified $R(6,6)/R(1,6)$. However, no explicit formula for $\frak M_n$ is known in general. An open question by MacWilliams and Sloane \cite[Research Problem (14.2)]{MacWilliams-Sloane-1977} asks how fast the number $\frak M_n$ grows with $n$. We will give an asymptotic formula for $\frak M_n$ as $n\to\infty$. Both $\frak N_{2,n}$ and $\frak M_n$ are important sequences; in the {\em On-line Encyclopedia of Integer sequences} \cite{oeis}, they are listed as A000214 and A001289, respectively.

The paper is organized as follows: Section 2 is a review of some mathematical results to be used in the paper. In Section 3, we derive an explicit formula for $\frak N_{q,n}$.  The asymptotic formula for $\frak M_n$ is proved in Section 4. We conclude the paper with a few brief remarks in Section 5 and a conclusion in Section 6. For readers' convenience, a list of notations used in the paper is complied in the appendix.

\section{Mathematical Background}

\subsection{Burnside's lemma}\

Let $G$ be a finite group acting on a finite set $X$. For $x\in X$, the subset $Gx=\{ax:a\in G\}\subset X$ is called the $G$-orbit of $x$. The $G$-orbits form a partition of $X$, and the number of $G$-orbits is given by the following formula referred to as Burnside's lemma:
\begin{equation}\label{bl}
\text{number of $G$-orbits}=\frac 1{|G|}\sum_{a\in G}\text{Fix}(a),
\end{equation}
where $\text{Fix}(a)=|\{x\in X:ax=x\}|$ is the number of fixed points of $a$ in $X$. If $a,b\in G$ are conjugate to each other, that is, $b=gag^{-1}$ for some $g\in G$, then $\text{Fix}(a)=\text{Fix}(b)$. Let $a_1,\dots,a_k$ be the representatives of the conjugacy classes of $G$ and let $[a_i]$ denote the conjugacy class of $a_i$. Then we have $|[a_i]|=|G|/|c(a_i)|$, where
\[
c(a_i)=\{g\in G:ga_i=a_ig\}
\]
is the {\em centralizer} of $a_i$ in $G$. Therefore \eqref{bl} can be more effectively computed as follows:
\begin{equation}\label{bl-1}
\text{number of $G$-orbits}=\frac 1{|G|}\sum_{i=1}^k|[a_i]|\text{Fix}(a_i)=\sum_{i=1}^k\frac{\text{Fix}(a_i)}{|c(a_i)|}.
\end{equation}

\subsection{Rational canonical form of a matrix}\

Let $\f$ be any field and let $f(X)=X^n+a_{n-1}X^{n-1}+\cdots+a_0\in\f[X]$ be a monic polynomial of degree $n$. A {\em companion matrix} of $f$ is an $n\times n$ matrix $A$ over $\f$ whose minimal polynomial is $f$; one can choose
\[
A=\left[
\begin{matrix}
0&1\cr
&0&1\cr
&&\cdot&\cdot\cr
&&&\cdot&\cdot\cr
&&&&0&1\cr
-a_0&\kern-0.5em -a_1&\cdot&\cdot&\cdot&\kern-0.5em -a_{n-1}
\end{matrix}\right].
\]
Every square matrix $A$ over $\f$ is similar (conjugate) to a {\em rational canonical form}
\[
\left[\begin{matrix}A_1\cr &\ddots\cr && A_m\end{matrix}\right],
\]
where each $A_i$ is a companion matrix of some $f_i\in\f[X]$ which is a power of an irreducible polynomial over $\f$. The polynomials $f_1,\dots,f_m$ are the {\em elementary divisors} of $A$. Two square matrices are similar if and only if they have the same list (multiset) of elementary divisors.

\subsection{Conjugacy classes of $\text{AGL}(n,\f_q)$}\

In this subsection, we describe the representatives of the conjugacy classes of $\text{AGL}(n,\f_q)$ and recall the formulas for the sizes of the centralizers of these representatives. These results can be found in \cite[\S 6.4]{Hou-ams-gsm-2018}.

A {\em partition} is a sequence of nonnegative integers $\lambda=(\lambda_1,\lambda_2,\dots)$ with only finitely many nonzero terms. We define $|\lambda|=\sum_{i\ge 1}i\lambda_i$ and $T(\lambda)=\{i:\lambda_i>0\}$. For example, if $\lambda=(2,0,1,3)$, then $|\lambda|=2\cdot 1+0\cdot 2+1\cdot 3+3\cdot 4=17$ and $T(\lambda)=\{1,3,4\}$. Let $\mathcal P$ denote the set of all partitions. Let $\mathcal I$ be the set of all monic irreducible polynomials in $\f_q[X]\setminus\{X\}$. For $f\in\mathcal I$ and $\lambda=(\lambda_1,\lambda_2,\dots)\in\mathcal P$, let $f^\lambda$ denote the multiset
\[
\{\underbrace{f^1,\dots,f^1}_{\lambda_1},\underbrace{f^2,\dots,f^2}_{\lambda_2},\dots\},
\]
i.e., the list with $\lambda_1$ copies of $f$, $\lambda_2$ copies of $f^2$, and so on.
Let $\sigma_{f^{\lambda}}$ be an element of $\text{GL}(|\lambda|\deg f,\f_q)$ with elementary divisors $f^\lambda$. For $\alpha=\left[\begin{smallmatrix}A&0\cr a&1\end{smallmatrix}\right]\in\text{AGL}(n_1,\f_q)$ and $\beta=\left[\begin{smallmatrix}B&0\cr b&1\end{smallmatrix}\right]\in\text{AGL}(n_2,\f_q)$, define
\[
\alpha\boxplus\beta=\left[\begin{matrix} A\cr &B\cr a&b&1\end{matrix}\right]\in\text{AGL}(n_1+n_2,\f_q).
\]
Let
\begin{equation}\label{2.6.0}
N_n=\left[\begin{matrix}
0&1\cr
&\cdot&\cdot\cr
&&\cdot&\cdot\cr
&&&\cdot&1\cr
&&&&0\end{matrix}\right]_{n\times n},\qquad J_n=I+N_n,
\end{equation}
where $I$ is the  identity matrix. For $\lambda=(\lambda_1,\lambda_2,\dots)\in\mathcal P$, let
\begin{equation}\label{sig-lam}
\sigma_\lambda=\underbrace{\left[\begin{smallmatrix} \displaystyle J_1\;&\vspace{1mm}\cr\displaystyle 0\;&\displaystyle 1\end{smallmatrix}\right]\boxplus\cdots\boxplus
\left[\begin{smallmatrix} \displaystyle J_1\;&\vspace{1mm}\cr\displaystyle 0\;&\displaystyle 1\end{smallmatrix}\right]}_{\lambda_1}\boxplus
\underbrace{\left[\begin{smallmatrix} \displaystyle J_2\;&\vspace{1mm}\cr\displaystyle 0\;&\displaystyle 1\end{smallmatrix}\right]\boxplus\cdots\boxplus
\left[\begin{smallmatrix} \displaystyle J_2\;&\vspace{1mm}\cr\displaystyle 0\;&\displaystyle 1\end{smallmatrix}\right]}_{\lambda_2}\boxplus\cdots\in\text{AGL}(|\lambda|,F)
\end{equation}
and, for $t\in T(\lambda)$,
\begin{align}\label{sig-lam-t}
\sigma_{\lambda,t}=\,& \underbrace{\left[\begin{smallmatrix} \displaystyle J_1\;&\vspace{1mm}\cr\displaystyle 0\;&\displaystyle 1\end{smallmatrix}\right]\boxplus\cdots\boxplus
\left[\begin{smallmatrix} \displaystyle J_1\;&\vspace{1mm}\cr\displaystyle 0\;&\displaystyle 1\end{smallmatrix}\right]}_{\lambda_1}\boxplus\cdots\boxplus
\underbrace{\left[\begin{smallmatrix} \displaystyle J_{t-1}\;&\vspace{1mm}\cr\displaystyle 0\;&\displaystyle 1\end{smallmatrix}\right]\boxplus\cdots\boxplus
\left[\begin{smallmatrix} \displaystyle J_{t-1}\;&\vspace{1mm}\cr\displaystyle 0\;&\displaystyle 1\end{smallmatrix}\right]}_{\lambda_{t-1}}\\
&\boxplus\underbrace{\left[\begin{smallmatrix} \displaystyle J_t\;&\vspace{0.7mm}\cr\displaystyle \epsilon_t\;&\displaystyle 1\end{smallmatrix}\right]\boxplus
\left[\begin{smallmatrix} \displaystyle J_t\;&\vspace{1mm}\cr\displaystyle 0\;&\displaystyle 1\end{smallmatrix}\right]\boxplus\cdots\boxplus
\left[\begin{smallmatrix} \displaystyle J_t\;&\vspace{1mm}\cr\displaystyle 0\;&\displaystyle 1\end{smallmatrix}\right]}_{\lambda_t}\cr
&\boxplus \underbrace{\left[\begin{smallmatrix} \displaystyle J_{t+1}\;&\vspace{1mm}\cr\displaystyle 0\;&\displaystyle 1\end{smallmatrix}\right]\boxplus\cdots\boxplus
\left[\begin{smallmatrix} \displaystyle J_{t+1}\;&\vspace{1mm}\cr\displaystyle 0\;&\displaystyle 1\end{smallmatrix}\right]}_{\lambda_{t+1}}\boxplus\cdots\in \text{AGL}(|\lambda|,F),\nonumber
\end{align}
where 
\[
\epsilon_t=(1,0,\dots,0)\in\f_q^t.
\] 
By \cite[Theorem~6.23]{Hou-ams-gsm-2018}, a set of representatives of the conjugacy classes of $\text{AGL}(n,\f_q)$ is given by $\mathcal C=\mathcal C_1\cup\mathcal C_2$, where
\begin{equation}\label{C1}
\mathcal C_1=\Bigl\{ \sigma_\lambda\boxplus\Bigl(\bboxplus_{f\in\mathcal I}\sigma_{f^{\lambda_f}}\Bigr):\lambda,\lambda_f\in\mathcal P,\ |\lambda|+\sum_{f\in\mathcal I}|\lambda_f|\deg f=n\Bigr\},
\end{equation}
\begin{equation}\label{C2}
\mathcal C_2=\Bigl\{ \sigma_{\lambda,t}\boxplus\Bigl(\bboxplus_{f\in\mathcal I}\sigma_{f^{\lambda_f}}\Bigr):\lambda,\lambda_f\in\mathcal P,\ |\lambda|>0,\ t\in T(\lambda),\ |\lambda|+\sum_{f\in\mathcal I}|\lambda_f|\deg f=n\Bigr\}.
\end{equation}
We shall refine the description of $\mathcal C_1$ and $\mathcal C_2$ to serve the purpose of the present paper.
For $f\in\mathcal I$, the {\em order} of $f$, denoted by $\text{ord}\,f$, is the multiplicative order of the roots of $f$. If $\text{ord}\,f=d$, then $\deg f=o_d(q)$, the multiplicative order of $q$ in $\Bbb Z/d\Bbb Z$. For $d\ge 1$ with $p\nmid d$, where $p=\text{char}\,\f_q$, let $I_d=\{f\in\mathcal I: \text{ord}\,f=d\}$. Then $|I_d|=\psi(d):=\phi(d)/o_d(q)$, where $\phi$ is the Euler totient function. We order the partitions in the following manner: For $\lambda=(\lambda_1,\lambda_2,\dots), \eta=(\eta_1,\eta_2,\dots)\in\mathcal P$, ``$\lambda<\eta$'' means that $|\lambda|<|\eta|$, or $|\lambda|=|\eta|$ and for the largest $i$ such that $\lambda_i\ne\eta_i$ we have $\lambda_i<\eta_i$. (Partitions in this particular order can be easily generated by computer.) Let $D=\{d>1:d \mid q^i-1\ \text{for some}\ 1\le i\le n\}$.   For $d\in D$, let 
\[
\Lambda_d=\{(\boldsymbol{\lambda}^{(1)},\dots,\boldsymbol{\lambda}^{(\psi(d))}):\boldsymbol{\lambda}^{(i)}\in\mathcal P,\ \boldsymbol{\lambda}^{(1)}\le\cdots\le\boldsymbol{\lambda}^{(\psi(d))}\},
\]
and for $\boldsymbol{\lambda}=(\boldsymbol{\lambda}^{(1)},\dots,\boldsymbol{\lambda}^{(\psi(d))})\in\Lambda_d$, let $|\boldsymbol{\lambda}|=\sum_{i=1}^{\psi(d)}|\boldsymbol{\lambda}^{(i)}|$.
Let
\[
\Omega=\Bigl\{(\lambda,(\boldsymbol{\lambda}_d)_{d\in D}):\lambda\in\mathcal P,\, \boldsymbol{\lambda}_d\in\Lambda_d,\,|\lambda|+\!\sum_{d\in D}\!o_d(q)|\boldsymbol{\lambda}_d|=n\Bigr\}.
\]
Then
\begin{align}\label{C1-1}
\mathcal C_1=\bigcup_{(\lambda,(\boldsymbol{\lambda}_d)_{d\in D})\in\Omega}\Bigl\{ &\sigma_\lambda\boxplus\Bigl(\bboxplus_{d\in D}\Bigl(\bboxplus_{f\in I_d}\sigma_{f^{\lambda_f}} \Bigr) \Bigr):\\
&\lambda_f\in\mathcal P,\ (\lambda_f)_{f\in I_d}\ \text{is a permutation of}\ \boldsymbol{\lambda}_d\Bigr\} \nonumber
\end{align}
and
\begin{align}\label{C2-1}
\mathcal C_2=\bigcup_{\substack{(\lambda,(\boldsymbol{\lambda}_d)_{d\in D})\in\Omega\cr |\lambda|>0}}\Bigl\{ &\sigma_{\lambda,t}\boxplus\Bigl(\bboxplus_{d\in D}\Bigl(\bboxplus_{f\in I_d}\sigma_{f^{\lambda_f}} \Bigr) \Bigr):\\
&t\in T(\lambda),\ \lambda_f\in\mathcal P,\ (\lambda_f)_{f\in I_d}\ \text{is a permutation of}\ \boldsymbol{\lambda}_d\Bigr\}. \nonumber
\end{align}
In $\mathcal C_1$, let
\begin{equation}\label{2.10}
\alpha=\sigma_\lambda\boxplus\Bigl(\bboxplus_{d\in D}\Bigl(\bboxplus_{f\in I_d}\sigma_{f^{\lambda_f}} \Bigr) \Bigr),
\end{equation}
and in $\mathcal C_2$, let
\begin{equation}\label{2.12}
\beta=\sigma_{\lambda,t}\boxplus\Bigl(\bboxplus_{d\in D}\Bigl(\bboxplus_{f\in I_d}\sigma_{f^{\lambda_f}} \Bigr) \Bigr).
\end{equation}
The sizes of the centralizers of $\alpha$ and $\beta$ in $\text{AGL}(n,\f_q)$ are given by \cite[Theorem~6.24]{Hou-ams-gsm-2018}. 
Write $\lambda=(\lambda_1,\lambda_2,\dots)$ and $\boldsymbol{\lambda}_d=(\boldsymbol{\lambda}_d^{(1)},\dots,\boldsymbol{\lambda}_d^{(\psi(d))})$, where $\boldsymbol{\lambda}_d^{(i)}=(\boldsymbol{\lambda}_{d,1}^{(i)},\boldsymbol{\lambda}_{d,2}^{(i)},\dots)\in\mathcal P$.
We have 
\begin{align}\label{2.13}
&|c(\alpha)|:=a_1(\lambda,(\boldsymbol{\lambda}_d)_{d\in D})=\\
&q^{\sum_j\lambda_j+\sum_{j,k}\min(j,k)\lambda_j\lambda_k+\sum_{d\in D}o_d(q)\sum_{i=1}^{\psi(d)}\sum_{j,k}\min(j,k)\boldsymbol{\lambda}^{(i)}_{d,j}\boldsymbol{\lambda}^{(i)}_{d,k}}\cr
&\cdot\Bigl(\prod_j\prod_{l=1}^{\lambda_j}(1-q^{-l})\Bigr)\Bigl(\prod_{d\in D}\prod_{i=1}^{\psi(d)}\prod_j\prod_{u=1}^{\boldsymbol{\lambda}^{(i)}_{d,j}}(1-q^{-o_d(q)u})\Bigr), \nonumber
\end{align}
\begin{align}\label{2.14}
&|c(\beta)|:=a_2(t,\lambda,(\boldsymbol{\lambda}_d)_{d\in D})=\\
&\frac1{q^{\lambda_t}-1}q^{\sum_{j\ge t}\lambda_j+\sum_{j,k}\min(j,k)\lambda_j\lambda_k+\sum_{d\in D}o_d(q)\sum_{i=1}^{\psi(d)}\sum_{j,k}\min(j,k)\boldsymbol{\lambda}^{(i)}_{d,j}\boldsymbol{\lambda}^{(i)}_{d,k}}\cr
&\cdot\Bigl(\prod_j\prod_{l=1}^{\lambda_j}(1-q^{-l})\Bigr)\Bigl(\prod_{d\in D}\prod_{i=1}^{\psi(d)}\prod_j\prod_{u=1}^{\boldsymbol{\lambda}^{(i)}_{d,j}}(1-q^{-o_d(q)u})\Bigr). \nonumber
\end{align}

\noindent{\bf Example.}
In \eqref{2.10} and \eqref{2.12}, let $q=2$, $\lambda=(3,0,1)$, $t=3$ and 
\[
\bboxplus_{d\in D}\Bigl(\bboxplus_{f\in I_d}\sigma_{f^{\lambda_f}} \Bigr)=\sigma_{f_1}^{\lambda_{f_1}}\boxplus \sigma_{f_2}^{\lambda_{f_1}},
\]
where $f_1=X^3+X+1$, $f_2=X^3+X^2+1\in\mathcal I$, $\lambda_{f_1}=(1,2)$ and $\lambda_{f_2}=(2,0,1)$. Let's see what $\alpha$ and $\beta$ look like. We have $\text{ord}\,f_1=\text{ord}\,f_2=7$, i.e., $f_1,f_2\in I_7$. Hence ${\boldsymbol\lambda}_7=({\boldsymbol\lambda}_7^{(1)},{\boldsymbol\lambda}_7^{(2)})$, where ${\boldsymbol\lambda}_7^{(1)}=(1,2)$ and ${\boldsymbol\lambda}_7^{(2)}=(2,0,1)$, and ${\boldsymbol\lambda}_d=((0),\dots,(0))$ for $7\ne d\in D$. Let $A_i^{(j)}$ be a companion matrix of $f_i^j$. Then
\begin{align*}
\alpha=\,&\left[\begin{smallmatrix} \displaystyle J_1\;&\vspace{1mm}\cr\displaystyle 0\;&\displaystyle 1\end{smallmatrix}\right]
\boxplus 
\left[\begin{smallmatrix} \displaystyle J_1\;&\vspace{1mm}\cr\displaystyle 0\;&\displaystyle 1\end{smallmatrix}\right]
\boxplus
\left[\begin{smallmatrix} \displaystyle J_1\;&\vspace{1mm}\cr\displaystyle 0\;&\displaystyle 1\end{smallmatrix}\right]
\boxplus
\left[\begin{smallmatrix} \displaystyle J_3\;&\vspace{1mm}\cr\displaystyle 0\;&\displaystyle 1\end{smallmatrix}\right]\cr
&\boxplus
\left[\begin{smallmatrix} \displaystyle A_1^{(1)}\;&\vspace{1mm}\cr\displaystyle 0\;&\displaystyle 1\end{smallmatrix}\right]
\boxplus
\left[\begin{smallmatrix} \displaystyle A_1^{(2)}\;&\vspace{1mm}\cr\displaystyle 0\;&\displaystyle 1\end{smallmatrix}\right]
\boxplus
\left[\begin{smallmatrix} \displaystyle A_1^{(2)}\;&\vspace{1mm}\cr\displaystyle 0\;&\displaystyle 1\end{smallmatrix}\right]
\boxplus
\left[\begin{smallmatrix} \displaystyle A_2^{(1)}\;&\vspace{1mm}\cr\displaystyle 0\;&\displaystyle 1\end{smallmatrix}\right]
\boxplus
\left[\begin{smallmatrix} \displaystyle A_2^{(1)}\;&\vspace{1mm}\cr\displaystyle 0\;&\displaystyle 1\end{smallmatrix}\right]
\boxplus
\left[\begin{smallmatrix} \displaystyle A_2^{(3)}\;&\vspace{1mm}\cr\displaystyle 0\;&\displaystyle 1\end{smallmatrix}\right],
\end{align*}
\begin{align*}
\beta=\,&\left[\begin{smallmatrix} \displaystyle J_1\;&\vspace{1mm}\cr\displaystyle 0\;&\displaystyle 1\end{smallmatrix}\right]
\boxplus 
\left[\begin{smallmatrix} \displaystyle J_1\;&\vspace{1mm}\cr\displaystyle 0\;&\displaystyle 1\end{smallmatrix}\right]
\boxplus
\left[\begin{smallmatrix} \displaystyle J_1\;&\vspace{1mm}\cr\displaystyle 0\;&\displaystyle 1\end{smallmatrix}\right]
\boxplus
\left[\begin{smallmatrix} \displaystyle J_3\;&\vspace{1mm}\cr\displaystyle \epsilon_3\;&\displaystyle 1\end{smallmatrix}\right]\cr
&\boxplus
\left[\begin{smallmatrix} \displaystyle A_1^{(1)}\;&\vspace{1mm}\cr\displaystyle 0\;&\displaystyle 1\end{smallmatrix}\right]
\boxplus
\left[\begin{smallmatrix} \displaystyle A_1^{(2)}\;&\vspace{1mm}\cr\displaystyle 0\;&\displaystyle 1\end{smallmatrix}\right]
\boxplus
\left[\begin{smallmatrix} \displaystyle A_1^{(2)}\;&\vspace{1mm}\cr\displaystyle 0\;&\displaystyle 1\end{smallmatrix}\right]
\boxplus
\left[\begin{smallmatrix} \displaystyle A_2^{(1)}\;&\vspace{1mm}\cr\displaystyle 0\;&\displaystyle 1\end{smallmatrix}\right]
\boxplus
\left[\begin{smallmatrix} \displaystyle A_2^{(1)}\;&\vspace{1mm}\cr\displaystyle 0\;&\displaystyle 1\end{smallmatrix}\right]
\boxplus
\left[\begin{smallmatrix} \displaystyle A_2^{(3)}\;&\vspace{1mm}\cr\displaystyle 0\;&\displaystyle 1\end{smallmatrix}\right].
\end{align*}
Note that $n=|\lambda|+3|\boldsymbol{\lambda}_7|=6+3(5+5)=36$, so $\alpha,\beta\in\text{AGL}(36,\f_2)$. By \eqref{2.13} and \eqref{2.14} (with $o_7(2)=3$),
\begin{align*}
|c(\alpha)|=\,&2^{4+3^2+2\cdot 3+3\cdot 1+3(1+2\cdot 2+2\cdot2^2+2^2+2\cdot 2+3\cdot 1)} \cr
&\cdot(1-2^{-1})(1-2^{-2})(1-2^{-3})(1-2^{-1})(1-2^{-3\cdot 1})(1-2^{-3\cdot 1})(1-2^{-3\cdot 2})\cr
&\cdot (1-2^{-3\cdot 1})(1-2^{-3\cdot 2})(1-2^{-3\cdot 1})\cr
=\,&2^{63}\cdot 3^5\cdot 7^7,
\end{align*}
\begin{align*}
|c(\beta)|=\,&\frac 1{2^1-1} 2^{1+3^2+2\cdot 3+3\cdot 1+3(1+2\cdot 2+2\cdot2^2+2^2+2\cdot 2+3\cdot 1)} \cr
&\cdot(1-2^{-1})(1-2^{-2})(1-2^{-3})(1-2^{-1})(1-2^{-3\cdot 1})(1-2^{-3\cdot 1})(1-2^{-3\cdot 2})\cr
&\cdot (1-2^{-3\cdot 1})(1-2^{-3\cdot 2})(1-2^{-3\cdot 1})\cr
=\,&2^{60}\cdot 3^5\cdot 7^7.
\end{align*}

\section{A Formula for $\frak N_{q,n}$}

The objective of this section is to derive an explicit formula for $\frak N_{q,n}$, the number of $\text{AGL}(n,\f_q)$ orbits of $\mathcal F(\f_q^n,\f_q)$. Although parts of the proof are rather technical, the general strategy is quite simple.

\subsection{Strategy}\

We shall identify $\text{GL}(n,\f_q)$ with the subgroup $\{\left[\begin{smallmatrix}A&0\cr 0&1\end{smallmatrix}\right]:A\in\text{GL}(n,\f_q)\}$ of $\text{AGL}(n,\f_q)$. For $\gamma=\left[\begin{smallmatrix}A&0\cr a&1\end{smallmatrix}\right]\in\text{AGL}(n,\f_q)$, we can identify $\gamma$ with the affine map $\f_q^n\to\f_q$, $x\mapsto xA+a$, whence $\gamma(f)=f\circ\gamma$ for $f\in\mathcal F(\f_q^n,\f_q)$. For $\gamma\in\text{AGL}(n,\f_q)$, define
\begin{equation}\label{2.1}   
\text{Fix}(\gamma)=|\{f\in\mathcal F(\f_q^n,\f_q):\gamma(f)=f\}|
\end{equation}
and let $c(\gamma)$ denote the centralizer of $\gamma$ in $\text{AGL}(n,\f_q)$.
Let $\mathcal C$ be a set of representatives of the conjugacy classes of $\text{AGL}(n,\f_q)$. By Burnside's lemma,
\begin{equation}\label{2.4}
\frak N_{q,n}=\sum_{\gamma\in\mathcal C}\frac {\text{Fix}(\gamma)}{|c(\gamma)|}.
\end{equation}
We choose $\mathcal C=\mathcal C_1\cup\mathcal C_2$, where $\mathcal C_1$ and $\mathcal C_2$ are given in \eqref{C1-1} and \eqref{C2-1}, respectively. Then \eqref{2.4} becomes
\begin{align}\label{Nqn}
\frak N_{q,n}=\,&\sum_{\gamma\in\mathcal C_1}\frac {\text{Fix}(\gamma)}{|c(\gamma)|}+\sum_{\gamma\in\mathcal C_2}\frac {\text{Fix}(\gamma)}{|c(\gamma)|}\\
=\,&\sum_{(\lambda,(\boldsymbol{\lambda}_d)_{d\in D})\in\Omega}\sum_{\substack{(\lambda_f)_{f\in I_d},\ \lambda_f\in\mathcal P,\cr (\lambda_f)_{f\in I_d}\ \text{is a permutation of}\ \boldsymbol{\lambda}_d}}\frac{\text{Fix}(\alpha)}{|c(\alpha)|}\cr
&+\sum_{\substack{(\lambda,(\boldsymbol{\lambda}_d)_{d\in D})\in\Omega\cr t\in T(\lambda)}}\sum_{\substack{(\lambda_f)_{f\in I_d},\ \lambda_f\in\mathcal P,\cr (\lambda_f)_{f\in I_d}\ \text{is a permutation of}\ \boldsymbol{\lambda}_d}}\frac{\text{Fix}(\beta)}{|c(\beta)|}, \nonumber
\end{align}
where $\alpha$ and $\beta$ are given in \eqref{2.10} and \eqref{2.12}, respectively, and   $|c(\alpha)|$ and $|c(\beta)|$ are given in \eqref{2.13} and \eqref{2.14}, respectively. 

It remains to determine $\text{Fix}(\alpha)$ and $\text{Fix}(\beta)$. {\em For entire Section 3, symbols $\alpha$ and $\beta$ are reserved for the elements of $\text{\rm AGL}(n,\f_q)$ defined in \eqref{2.10} and \eqref{2.12}.}

\subsection{Number of fixed points of $\alpha$ and $\beta$}\

The objective of this subsection is to determine $\text{Fix}(\alpha)$ and $\text{Fix}(\beta)$.
In general, for $\gamma\in\text{AGL}(n,\f_q)$ and $f\in\mathcal F(\f_q^n,\f_q)$, $\gamma(f)=f$ if and only if $f\circ\gamma=f$; this happens if and only if $f$ is constant on every $\langle\gamma\rangle$-orbit in $\f_q^n$, where $\langle\gamma\rangle$ is the cyclic group generated by $\gamma$. Hence
\begin{equation}\label{2.24}
\text{Fix}(\gamma)=q^{e(\gamma)},
\end{equation}
where $e(\gamma)$ is the number of $\langle\gamma\rangle$-orbits in $\f_q^n$. Let
\begin{equation}\label{fix}
\text{fix}(\gamma)=|\{x\in\f_q^n:\gamma(x)=x\}|.
\end{equation}
(We remind the reader that $\text{fix}(\gamma)$ is the number fixed points of $\gamma$ in $\f_q^n$, while $\text{Fix}(\gamma)$ is the number fixed points of $\gamma$ in $\mathcal F(\f_q^n,\f_q)$.)
By Burnside's lemma (yes, another application of Burnside's lemma),
\[
e(\gamma)=\frac 1{o(\gamma)}\sum_{k=1}^{o(\gamma)}\text{fix}(\gamma^k).
\]
Note that $\text{fix}(\gamma^k)=\text{fix}(\gamma^{\text{gcd}(k,o(\gamma))})$ and that for each $l\mid o(\gamma)$, the number of $k$ ($1\le k\le o(\gamma)$) such that $\text{gcd}(k,o(\gamma))=l$ is $\phi(o(\gamma)/l)$. Hence
\begin{equation}\label{2.25}
e(\gamma)=\frac 1{o(\gamma)}\sum_{k\mid o(\gamma)}\phi(o(\gamma)/k)\,\text{fix}(\gamma^k).
\end{equation}
In the next four lemmas, we first compute $o(\gamma)$ and $\text{fix}(\gamma^k)$ when $\gamma$ is a component of $\alpha$ or $\beta$, then we determine $o(\alpha)$ and $o(\beta)$, and finally we determine $\text{Fix}(\alpha)$ and $\text{Fix}(\beta)$. We will see that $\text{Fix}(\alpha)$ depends only on $(\lambda,(\boldsymbol{\lambda}_d)_{d\in D})$ and $\text{Fix}(\beta)$ depends only on $(t,\lambda,(\boldsymbol{\lambda}_d)_{d\in D})$.

Let $o(\ )$ denote the order of a group element and $\nu(\ )$ denote the $p$-adic order of integers, where $p=\text{char}\,\f_q$.

\begin{lem}\label{L2.1}
Let $f\in I_d$, where $d\in D\cup\{1\}$, and let $A\in\text{\rm GL}(o_d(q)l,\f_q)$ be a companion matrix of $f^l$, treated as an element of $\text{\rm AGL}(o_d(q)l,\f_q)$. Then $o(A)=dp^{\lceil\log_pl\rceil}$ and 
\begin{equation}\label{2.17}
\text{\rm fix}(A^k)=q^{\min(p^{\nu(k)},l)\epsilon(d,k)},\qquad  k\ge0,
\end{equation}
where
\begin{equation}\label{2.18}
\epsilon(d,k)=\begin{cases}
o_d(q)&\text{if}\ d\mid k,\cr
0&\text{if}\ d\nmid k.
\end{cases}
\end{equation}
\end{lem}

\begin{proof}
Write $k=p^{\nu(k)}k_1$, where $p\nmid k_1$. Then by \cite[Lemma~6.11]{Hou-ams-gsm-2018}, the nullity of $A^k-I$ is
\begin{align*}
\text{null}(A^k-I)&=\deg\text{gcd}(X^k-1,f^l)=\deg\text{gcd}\bigl((X^{k_1}-1)^{p^{\nu(k)}},f^l\bigr)\cr
&=\min(p^{\nu(k)},l)\deg\text{gcd}(X^{k_1}-1,f)\cr
&=\min(p^{\nu(k)},l)\epsilon(d,k),
\end{align*}
which gives \eqref{2.17}. Note that 
\[
\begin{array}{rcl}
A^k=I &\Leftrightarrow & \min(p^{\nu(k)},l)\epsilon(d,k)=o_d(q)l\cr
&\Leftrightarrow & d\mid k\ \text{and}\ p^{\nu(k)}\ge l\cr
&\Leftrightarrow & d\mid k\ \text{and}\ \nu(k)\ge \log_pl.
\end{array}
\]
Hence $o(A)=dp^{\lceil\log_pl\rceil}$.
\end{proof}

\begin{lem}\label{L2.2}
Let $\sigma=\left[\begin{smallmatrix}J_m&0\cr \epsilon_m&1\end{smallmatrix}\right]\in\text{\rm AGL}(m,\f_q)$. Then $o(\sigma)=p^{1+\lfloor\log_p m\rfloor}$ and for $k\ge 0$,
\begin{equation}\label{2.19}
\text{\rm fix}(\sigma^k)=\begin{cases}
q^m&\text{if}\ \nu(k)\ge 1+\lfloor\log_p m\rfloor,\cr
0&\text{if}\ \nu(k)< 1+\lfloor\log_p m\rfloor.
\end{cases}
\end{equation}
\end{lem}

\begin{proof}
We have
\[
\sigma^k=\left[\begin{matrix} J_m^k&0\cr \epsilon_m(I+J_m+\cdots+J_m^{k-1})& 1\end{matrix}\right],
\]
where
\[
J_m^k=(I+N_m)^k=\sum_{i=0}^k\binom kiN_m^i
\]
and
\[
I+J_m+\cdots+J_m^{k-1}=\sum_{i=0}^{k-1}\sum_{j=0}^i\binom ijN_m^j=\sum_{j=0}^{k-1}\Bigl(\sum_{i=j}^{k-1}\binom ij\Bigr)N_m^j=\sum_{j=0}^{k-1}\binom k{j+1}N_m^j.
\]
In the above, $J_m^k=I$ if and only if $\binom k1=\cdots=\binom k{m-1}=0$ and $\epsilon_m(I+J_m+\cdots+J_m^{k-1})=0$ if and only if $\binom k1=\cdots=\binom km=0$. Let $\text{id}$ denote the identity of $\text{AGL}(m,\f_2)$. Then
\begin{align*}
\sigma^k=\text{id} &\Leftrightarrow \binom k1=\cdots=\binom km=0 \Leftrightarrow p^{\nu(k)}>m \cr
&\Leftrightarrow \nu(k)>\log_p m  \Leftrightarrow \nu(k)\ge 1+\lfloor\log_p m\rfloor,
\end{align*}
so $o(\sigma)=p^{1+\lfloor\log_p m\rfloor}$.

For $x\in\f_q^m$, the equation $\sigma^k(x)=x$ is equivalent to
\[
\epsilon_m(I+J_m+\cdots+J_m^{k-1})=-x(J_m^k-I),
\]
i.e.,
\[
(\textstyle\binom k1, \textstyle\binom k2,\dots, \textstyle\binom km)=-x\left[
\begin{matrix}0&\binom k1&\cdots&\binom k{m-1}\cr
&\ddots&\ddots&\vdots\cr
&&\ddots&\binom k1\cr
&&&0
\end{matrix}\right].
\]
This holds if and only if $\binom k1=\cdots=\binom km=0$, i.e, $\nu(k)\ge 1+\lfloor\log_p m\rfloor$. Hence we have \eqref{2.19}.
\end{proof}

For $\lambda=(\lambda_1,\lambda_2,\dots)\in\mathcal P$, define $\frak m(\lambda)=\max\{i:\lambda_i>0\}$. For $\boldsymbol{\lambda}=(\boldsymbol{\lambda}^{(1)},\dots,$ $\boldsymbol{\lambda}^{(\psi(d))})\in\Lambda_d$, where $d\in D$, $\boldsymbol{\lambda}^{(i)}\in\mathcal P$, define $\frak m(\boldsymbol{\lambda})=\max_{1\le i\le\psi(d)}\frak m(\boldsymbol{\lambda}^{(i)})$.

\begin{lem}\label{L2.3}
{\rm (i)} We have 
\begin{equation}\label{2.20}
o(\alpha):=b_1(\lambda,(\boldsymbol{\lambda}_d)_{d\in D})=\text{\rm lcm}\{d:|\boldsymbol{\lambda}_d|>0\}\, p^{\lceil \log_p\max(\{\frak m(\lambda)\}\cup\{\frak m(\boldsymbol{\lambda}_d):d\in D\})\rceil}.
\end{equation}
{\rm (}We define $\text{\rm lcm}(\emptyset)=1$.{\rm )}

\medskip

{\rm (ii)} We have 
\begin{align}\label{2.21}
&o(\beta):=b_2(t,\lambda,(\boldsymbol{\lambda}_d)_{d\in D})\\
&=b_1(\lambda,(\boldsymbol{\lambda}_d)_{d\in D})\cdot
\begin{cases}
p&\text{if $t=\frak m(\lambda)\ge\max\{\frak m(\boldsymbol{\lambda}_d):d\in D\}$ and $t$ is a power of $p$},\cr
1&\text{otherwise}.
\end{cases}\nonumber
\end{align}
\end{lem}

\begin{proof}
(i) By \eqref{2.10}, $\alpha$ is a direct sum of affine transformations $\sigma_\lambda$ and $\sigma_{f^{\lambda_f}}$ ($f\in I_d, d\in D$). Therefore, the order of $\alpha$ equals the lcm of the orders of its components, that is,
\begin{equation}\label{L2.3-eq1}
o(\alpha)=\text{lcm}\Bigl(\{o(\sigma_\lambda)\}\cup\Bigl(\bigcup_{d\in D}\{o(\sigma_{f^{\lambda_f}}):f\in I_d\}\Bigr)\Bigr).
\end{equation}
By \eqref{sig-lam}, $o(\sigma_\lambda)=\text{lcm}\{o(J_i):\lambda_i>0\}$, where $\lambda=(\lambda_1,\lambda_2,\dots)$, and by Lemma~\ref{L2.1} (with $f^l=(x-1)^i$), $o(J_i)=p^{\lceil\log_pi\rceil}$. Thus
\begin{equation}\label{L2.3-eq2}
o(\sigma_\lambda)=p^{\max\{\lceil\log_pi\rceil:\lambda_i>0\}}=p^{\lceil\log_p\frak m(\lambda)\rceil}.
\end{equation}
For $d\in D$ and $f\in I_d$, write $\lambda_f=(\lambda_{f,1},\lambda_{f,2},\dots)$. Since the elementary divisors of $\sigma_{f^{\lambda_f}}$ are
\[
\underbrace{f^1,\dots,f^1}_{\lambda_{f,1}},\, \underbrace{f^2,\dots,f^2}_{\lambda_{f,2}},\dots,
\]
$o(\sigma_{f^{\lambda_f}})=\text{lcm}\{o(A_i):\lambda_{f,i}>0\}$, where $A_i$ is a companion matrix of $f^i$ and $o(A_i)=dp^{\lceil\log_pi\rceil}$ by Lemma~\ref{L2.1}. Thus
\[
o(\sigma_{f^{\lambda_f}})=dp^{\max\{\lceil\log_pi\rceil:\lambda_{f,i}>0\}}=dp^{\lceil\log_p\frak m(\lambda_f)\rceil}.
\]
Since $(\lambda_f)_{f\in I_d}$ is a permutation of $\boldsymbol{\lambda}_d$, we have 
\begin{equation}\label{L2.3-eq3}
\text{lcm}\{o(\sigma_{f^{\lambda_f}}):f\in I_d\}=dp^{\lceil\log_p\frak m(\boldsymbol{\lambda}_d)\rceil}.
\end{equation}
Combining \eqref{L2.3-eq1} -- \eqref{L2.3-eq3} gives
\[
o(\alpha)=\text{\rm lcm}\{d:|\boldsymbol{\lambda}_d|>0\}\, p^{\lceil \log_p\max(\{\frak m(\lambda)\}\cup\{\frak m(\boldsymbol{\lambda}_d):d\in D\})\rceil}.
\]

\medskip
(ii) By \eqref{sig-lam-t},
\[
o(\sigma_{\lambda,t})=\text{lcm}\Bigl(\{o(J_i):i\ne t,\;\lambda_i>0\}\cup\Bigl\{o\Bigl( \left[\begin{smallmatrix} \displaystyle J_t\;&\vspace{0.7mm}\cr\displaystyle \epsilon_t\;&\displaystyle 1\end{smallmatrix}\right]  \Bigr)\Bigr\}\Bigr),
\]
where $o(J_i)=p^{\lceil\log_pi\rceil}$ (Lemma~\ref{L2.1}) and $o( \left[\begin{smallmatrix}  J_t&\cr \epsilon_t& 1\end{smallmatrix}\right])=p^{1+\lfloor\log_pt\rfloor}$ (Lemma~\ref{L2.2}). Hence
\begin{equation}\label{L2.3-eq4}
o(\sigma_{\lambda,t})=p^{\max(\{1+\lfloor\log_pt\rfloor\}\cup\{\lceil\log_pi\rceil:i\ne t,\,\lambda_i>0\})}.
\end{equation}
Now by \eqref{2.12}, \eqref{L2.3-eq3} and \eqref{L2.3-eq4},
\begin{equation}\label{L2.3-eq5}
o(\beta)=\text{lcm}\Bigl(\{o(\sigma_{\lambda,t})\}\cup\Bigl(\bigcup_{d\in D}\{o(\sigma_{f^{\lambda_f}}):f\in I_d\}\Bigr)\Bigr)=\text{lcm}\{d:|\boldsymbol{\lambda}_d|>0\}\, p^e,
\end{equation}
where
\[
e=\max(\{1+\lfloor\log_pt\rfloor\}\cup\{\lceil\log_pi\rceil:i\ne t,\lambda_i>0\}\cup\{\lceil \log_p\frak m(\boldsymbol{\lambda}_d)\rceil:d\in D\}).
\]
When $t<\max(\{\frak m(\lambda)\}\cup\{\frak m(\boldsymbol{\lambda}_d):d\in D\})$, 
\[
1+\lfloor\log_pt\rfloor\le \lceil\log_p\max(\{\frak m(\lambda)\}\cup\{{\frak m}(\boldsymbol{\lambda}_d):d\in D\} )\rceil,
\]
whence
\begin{equation}\label{L2.3-eq6}
e=\lceil\log_p\max(\{\frak m(\lambda)\}\cup\{\frak m(\boldsymbol{\lambda}_d):d\in D\} )\rceil.
\end{equation}
When $t=\max(\{\frak m(\lambda)\}\cup\{\frak m(\boldsymbol{\lambda}_d):d\in D\})$, that is, $t=\frak m(\lambda)\ge \max\{\frak m(\boldsymbol{\lambda}_d):d\in D\}$, 
\[
1+\lfloor\log_pt\rfloor\ge\max\{\lceil\log_pi\rceil:i\ne t,\;\lambda_i>0\}\cup\{\lceil \log_p\frak m(\boldsymbol{\lambda}_d)\rceil:d\in D\},
\]
whence
\begin{equation}\label{L2.3-eq7}
e=1+\lfloor\log_pt\rfloor=\begin{cases}
1+\lceil\log_pt\rceil&\text{if $t$ is a power of $p$},\cr
\lceil\log_pt\rceil&\text{otherwise}.
\end{cases}
\end{equation}
Combining \eqref{L2.3-eq5} -- \eqref{L2.3-eq7} gives \eqref{2.21}.
\end{proof}

We remind the reader that for $\lambda\in\mathcal P$, we write $\lambda=(\lambda_1,\lambda_2,\dots)$, and for $\boldsymbol{\lambda}_d\in\Lambda_d$, we write $\boldsymbol{\lambda}_d=(\boldsymbol{\lambda}_d^{(1)},\dots,\boldsymbol{\lambda}_d^{(\psi(d))})$, where $\boldsymbol{\lambda}_d^{(i)}=(\boldsymbol{\lambda}_{d,1}^{(i)},\boldsymbol{\lambda}_{d,2}^{(i)},\dots)\in\mathcal P$.

\begin{lem}\label{L2.4}
{\rm (i)} We have $\text{\rm Fix}(\alpha)=q^{e_1(\lambda,(\boldsymbol{\lambda}_d)_{d\in D})}$, where
\begin{align}\label{2.22}
&e_1(\lambda,(\boldsymbol{\lambda}_d)_{d\in D})=\\
&\frac 1{b_1(\lambda,(\boldsymbol{\lambda}_d)_{d\in D})}\sum_{k\mid b_1(\lambda,(\boldsymbol{\lambda}_d)_{d\in D})}\phi\bigl(b_1(\lambda,(\boldsymbol{\lambda}_d)_{d\in D})/k \bigr)\cr
&\cdot q^{\sum_{j\ge 1}\min(p^{\nu(k)},j)\lambda_j+\sum_{d\in D,\,d\mid k}o_d(q)\sum_{1\le i\le\psi(d)}\sum_{j\ge 1}\min(p^{\nu(k)},j)\boldsymbol{\lambda}_{d,j}^{(i)}}. \nonumber
\end{align}

{\rm (ii)}  We have $\text{\rm Fix}(\beta)=q^{e_2(t,\lambda,(\boldsymbol{\lambda}_d)_{d\in D})}$, where
\begin{align}\label{2.23}
&e_2(t,\lambda,(\boldsymbol{\lambda}_d)_{d\in D})=\\
&\frac 1{b_2(t,\lambda,(\boldsymbol{\lambda}_d)_{d\in D})}\sum_{\substack{k\mid b_2(t,\lambda,(\boldsymbol{\lambda}_d)_{d\in D})\cr \nu(k)\ge 1+\lfloor\log_pt\rfloor}}\phi\bigl(b_2(t,\lambda,(\boldsymbol{\lambda}_d)_{d\in D})/k \bigr)\cr
&\cdot q^{\sum_{j\ge 1}\min(p^{\nu(k)},j)\lambda_j+\sum_{d\in D,\,d\mid k}o_d(q)\sum_{1\le i\le\psi(d)}\sum_{j\ge 1}\min(p^{\nu(k)},j)\boldsymbol{\lambda}_{d,j}^{(i)}}. \nonumber
\end{align}
\end{lem}

\begin{proof}
(i) By \eqref{2.24}, it suffices to show that $e(\alpha)=e_1(\lambda,(\boldsymbol{\lambda}_d)_{d\in D})$, where $e(\alpha)$ is given in \eqref{2.25}. Recall from \eqref{2.25} that
\begin{equation}\label{e-alpha}
e(\alpha)=\frac 1{o(\alpha)}\sum_{k\mid o(\alpha)}\phi(o(\alpha)/k)\,\text{fix}(\alpha^k),
\end{equation}
where $o(\alpha)=b_1(\lambda,(\boldsymbol{\lambda}_d)_{d\in D})$.
In the above, by \eqref{2.10},
\[
\text{fix}(\alpha^k)=\text{fix}(\sigma_\lambda^k)\prod_{d\in D}\prod_{f\in I_d}\text{fix}(\sigma_{f^{\lambda_f}}^k),
\]
where $\text{fix}(\sigma_\lambda^k)$ and $\text{fix}(\sigma_{f^{\lambda_f}}^k)$ are computed as follows: By \eqref{sig-lam} and \eqref{2.17},
\[
\text{fix}(\sigma_\lambda^k)=\prod_{j\ge 1}\text{fix}(J_j^k)^{\lambda_j}=\prod_{j\ge 1}q^{\min(p^{\nu(k)},j)\lambda_j}.
\]
Let $A_j$ be a companion matrix of $f^j$ and $\lambda_f=(\lambda_{f,1},\lambda_{f,2},\dots)$. By \eqref{2.17},
\[
\text{fix}(\sigma_{f^{\lambda_f}}^k)=\prod_{j\ge 1}\text{fix}(A_j^k)^{\lambda_{f,j}}=\prod_{j\ge 1}q^{\min(p^{\nu(k)},j)\epsilon(d,k)\lambda_{f,j}}.
\]
Hence
\begin{align}\label{2.26}
\text{fix}(\alpha^k)\,
&=q^{\sum_{j\ge 1}\min(p^{\nu(k)},j)\lambda_j}\prod_{d\in D}\prod_{f\in I_d}q^{\sum_{j\ge 1}\min(p^{\nu(k)},j)\epsilon(d,k)\lambda_{f,j}}\\
&=q^{\sum_{j\ge 1}\min(p^{\nu(k)},j)\lambda_j+\sum_{d\in D}\sum_{1\le i\le\psi(d)}\sum_{j\ge 1}\min(p^{\nu(k)},j)\epsilon(d,k)\boldsymbol{\lambda}_{d,j}^{(i)}}\cr
&\kern 10.4em\text{(since $(\lambda_f)_{f\in I_d}$ is a permutation of $\boldsymbol{\lambda}_d$)}\cr
&=q^{\sum_{j\ge 1}\min(p^{\nu(k)},j)\lambda_j+\sum_{d\in D,\,d\mid k}o_d(q)\sum_{1\le i\le\psi(d)}\sum_{j\ge 1}\min(p^{\nu(k)},j)\boldsymbol{\lambda}_{d,j}^{(i)}}\cr
&\kern 23.2em \text{(by \eqref{2.18})}. \nonumber
\end{align}
Using \eqref{2.26} in \eqref{e-alpha} and comparing the result with \eqref{2.22}, we have $e(\alpha)=e_1(\lambda,(\boldsymbol{\lambda}_d)_{d\in D})$.

\medskip
(ii) It suffices to show that $e(\beta)=e_2(t,\lambda,(\boldsymbol{\lambda}_d)_{d\in D})$.
First, it follows from \eqref{2.19} and \eqref{2.17} that
\[
\text{fix}\Bigl(\left[\begin{smallmatrix} \displaystyle J_t\;& \displaystyle 0 \vspace{0.7mm}\cr\displaystyle \epsilon_t\;&\displaystyle 1\end{smallmatrix}\right]^k \Bigr)=
\begin{cases}
\text{fix}(J_t^k)&\text{if}\ \nu(k)\ge 1+\lfloor\log_pt\rfloor,\cr
0&\text{otherwise}.
\end{cases}
\]
Thus, by comparing $\sigma_\lambda$ and $\sigma_{\lambda,t}$ (\eqref{sig-lam} and \eqref{sig-lam-t}), we have
\[
\text{fix}(\sigma_{\lambda,t}^k)=
\begin{cases}
\text{fix}(\sigma_\lambda^k)&\text{if}\ \nu(k)\ge 1+\lfloor\log_pt\rfloor,\cr
0&\text{otherwise}.
\end{cases}
\]
Hence
\[
\text{fix}(\beta^k)=\text{fix}(\sigma_{\lambda,t}^k)\prod_{d\in D}\prod_{f\in I_d}\text{fix}(\sigma_{f^{\lambda_f}}^k)
=\begin{cases}
\text{fix}(\alpha^k)&\text{if}\ \nu(k)\ge 1+\lfloor\log_pt\rfloor,\cr
0&\text{otherwise}.
\end{cases}
\]
Therefore,
\begin{equation}\label{e-beta}
e(\beta)=\frac 1{o(\beta)}\sum_{k\mid o(\beta)}\phi(o(\beta)/k)\text{fix}(\beta^k)
=\frac 1{o(\beta)}\sum_{\substack{k\mid o(\beta)\cr \nu(k)\ge 1+\lfloor\log_pt\rfloor}}\phi(o(\beta)/k)\text{fix}(\alpha^k).
\end{equation}
In the above, $o(\beta)=b_2(t,\lambda,(\boldsymbol{\lambda}_d)_{d\in D})$ and $\text{fix}(\alpha^k)$ is given in \eqref{2.26}. Now comparing \eqref{e-beta} and \eqref{2.23} gives $e(\beta)=e_2(t,\lambda,(\boldsymbol{\lambda}_d)_{d\in D})$.
\end{proof}

\subsection{The formula for $\frak N_{q,n}$}\

We now assemble the formula for $\frak N_{q,n}$. 

For $\boldsymbol{\lambda}=(\boldsymbol{\lambda}^{(1)},\dots,\boldsymbol{\lambda}^{(\psi(d))})\in\Lambda_d$, let $s(\boldsymbol{\lambda})$ be the number of permutations of $(\boldsymbol{\lambda}^{(1)},\dots,\boldsymbol{\lambda}^{(\psi(d))})$; that is, if $(\boldsymbol{\lambda}^{(1)},\dots,\boldsymbol{\lambda}^{(\psi(d))})$ has $t$ distinct components with respective multiplicities $k_1,\dots,k_t$, then
\begin{equation}\label{s-lam}
s(\boldsymbol{\lambda})=\binom{\psi(d)}{k_1,\dots,k_t}=\frac{\psi(d)!}{k_1!\cdots k_t!}.
\end{equation}

\begin{thm}\label{T2.5}
We have
\[
\frak N_{q,n}=\sum_{(\lambda,(\boldsymbol{\lambda}_d)_{d\in D})\in\Omega}s(\boldsymbol{\lambda}_d)\frac{q^{e_1(\lambda,(\boldsymbol{\lambda}_d)_{d\in D})}}{a_1(\lambda,(\boldsymbol{\lambda}_d)_{d\in D})}+\sum_{\substack{(\lambda,(\boldsymbol{\lambda}_d)_{d\in D})\in\Omega\cr t\in T(\lambda)}}s(\boldsymbol{\lambda}_d)\frac{q^{e_2(t,\lambda,(\boldsymbol{\lambda}_d)_{d\in D})}}{a_2(t,\lambda,(\boldsymbol{\lambda}_d)_{d\in D})},
\]
where $s(\boldsymbol{\lambda}_d)$, $a_1(\lambda,(\boldsymbol{\lambda}_d)_{d\in D})$, $e_1(\lambda,(\boldsymbol{\lambda}_d)_{d\in D})$, $a_2(t,\lambda,(\boldsymbol{\lambda}_d)_{d\in D})$ and $e_2(t,\lambda,(\boldsymbol{\lambda}_d)_{d\in D})$ are given in \eqref{s-lam}, \eqref{2.13}, \eqref{2.22}, \eqref{2.14} and \eqref{2.23}, respectively.
\end{thm}

\begin{proof}
Recall from \eqref{Nqn} that 
\begin{align*}
\frak N_{q,n}
=\,&\sum_{(\lambda,(\boldsymbol{\lambda}_d)_{d\in D})\in\Omega}\sum_{\substack{(\lambda_f)_{f\in I_d},\ \lambda_f\in\mathcal P,\cr (\lambda_f)_{f\in I_d}\ \text{is a permutation of}\ \boldsymbol{\lambda}_d}}\frac{\text{Fix}(\alpha)}{|c(\alpha)|}\cr
&+\sum_{\substack{(\lambda,(\boldsymbol{\lambda}_d)_{d\in D})\in\Omega\cr t\in T(\lambda)}}\sum_{\substack{(\lambda_f)_{f\in I_d},\ \lambda_f\in\mathcal P,\cr (\lambda_f)_{f\in I_d}\ \text{is a permutation of}\ \boldsymbol{\lambda}_d}}\frac{\text{Fix}(\beta)}{|c(\beta)|}.
\end{align*}
In the above, by Lemma~\ref{L2.4}, \eqref{2.13} and \eqref{2.14}, 
\[
\frac{\text{Fix}(\alpha)}{|c(\alpha)|}=\frac{q^{e_1(\lambda,(\boldsymbol{\lambda}_d)_{d\in D})}}{a_1(\lambda,(\boldsymbol{\lambda}_d)_{d\in D})}
\]
and 
\[
\frac{\text{Fix}(\beta)}{|c(\beta)|}=\frac{q^{e_2(t,\lambda,(\boldsymbol{\lambda}_d)_{d\in D})}}{a_2(t,\lambda,(\boldsymbol{\lambda}_d)_{d\in D})}.
\]
Both expressions depend on $\boldsymbol{\lambda}_d$ rather than $(\lambda_f)_{f\in I_d}$. Hence
\[
\frak N_{q,n}=\sum_{(\lambda,(\boldsymbol{\lambda}_d)_{d\in D})\in\Omega}s(\boldsymbol{\lambda}_d)\frac{q^{e_1(\lambda,(\boldsymbol{\lambda}_d)_{d\in D})}}{a_1(\lambda,(\boldsymbol{\lambda}_d)_{d\in D})}+\sum_{\substack{(\lambda,(\boldsymbol{\lambda}_d)_{d\in D})\in\Omega\cr t\in T(\lambda)}}s(\boldsymbol{\lambda}_d)\frac{q^{e_2(t,\lambda,(\boldsymbol{\lambda}_d)_{d\in D})}}{a_2(t,\lambda,(\boldsymbol{\lambda}_d)_{d\in D})}.
\]
\end{proof}

\section{An Asymptotic Formula for $\frak M_n$}

Recall that $\frak M_n$ is the number of $\text{AGL}(n,\f_2)$ orbits of $R(n,n)/R(1,n)$. By \cite[Theorem~5.1]{Hou-JA-1995}, $\frak M_n$ is also the number of AGL orbits of $R(n-2,n)$. The main result of this section is the following asymptotic formula for $\frak M_n$ as $n\to\infty$.

\begin{thm}\label{T3.1}
We have
\begin{equation}\label{3.1}
\lim_{n\to\infty}\frak M_n\cdot\frac{\prod_{i=1}^{\infty}(1-2^{-i})}{2^{2^n-n^2-2n-1}}=1.
\end{equation}
\end{thm}

To prove this theorem, we need some preparatory results, mainly about compound matrices.

\subsection{Compound matrices and preparatory results}\

For $0\le r\le n$, let $\mathcal C^n_r$ denote the set of all subsets of $\{1,\dots,n\}$ of size $r$. Let $A$ be an $n\times n$ matrix (over any field). The $r$th {\em compound matrix} of $A$, denoted by $C_r(A)$, is the $\binom n r\times \binom nr$ matrix whose rows and columns are indexed by $\mathcal C^n_r$ and whose $(S,T)$-entry ($S,T\in\mathcal C^n_r$) is $\det A(S,T)$, where $A(S,T)$ is the submatrix of $A$ with row indices from $S$ and column indices from $T$. For general properties of compound matrices, see \cite[Chapter V]{Wedderburn-1964}. If the eigenvalues of $A$ are $\lambda_1,\dots,\lambda_n$ (counting multiplicity), then the eigenvalues of $C_r(A)$ are $\prod_{i\in S}\lambda_i$, $S\in\mathcal C^n_r$.

The quotient space $R(r,n)/R(r-1,n)$ has a basis $\{X_S:S\in\mathcal C^n_r\}$, where $X_S=\prod_{i\in S}X_i$. When $A\in\text{GL}(n,\f_2)$ acts on $R(r,n)/R(r-1,n)$, its matrix with respect to the basis $(X_S)_{S\in\mathcal C^n_r}$ of $R(r,n)/R(r-1,n)$, displayed in a row, is $C_r(A)$, i.e.,
\[
A((X_S)_{S\in\mathcal C^n_r})=(X_S)_{S\in\mathcal C^n_r} C_r(A);
\]
see \cite[\S 4]{Hou-DM-1996}. More generally, when $\sigma=\left[\begin{smallmatrix}A&0\cr a&1\end{smallmatrix}\right]\in\text{AGL}(n,\f_q)$ acts on $R(r,n)/R(s,n)$, $-1\le s<r\le n$, its matrix with respect to the basis $\{X_S:S\subset\{1,\dots,n\},\ s<|S|\le r\}$ is
\begin{equation}\label{CA}
\left[\begin{matrix}
C_r(A)&0&\cdots&0\cr
*&C_{r-1}(A)&\cdots&0\cr
\vdots &\vdots&\ddots&\vdots\cr
*&*&\cdots&C_{s+1}(A)\end{matrix}\right].
\end{equation}

\begin{lem}\label{L3.1}
Let $\f$ be a field and $\lambda_1,\dots,\lambda_n\in \f^*$ ($n\ge 1$) with $\lambda_1\ne 1$. Then for at least half of the subsets $S$ of $\{1,\dots,n\}$, $\prod_{i\in S}\lambda_i\ne 1$.
\end{lem}

\begin{proof}
For every $S\subset\{2,\dots,n\}$, at most one of $\prod_{i\in S}\lambda_i$ and $\prod_{i\in\{1\}\cup S}\lambda_i$ is $1$.
\end{proof} 

\begin{lem}\label{L3.2}
Let $A$ be an $m\times m$ matrix, $B$ be an $n\times n$ matrix, and $0\le k\le m$, $0\le l\le n$. Then $C_k(A)\otimes C_l(B)$ is a principal submatrix of $C_{k+l}(A\oplus B)$, where $A\oplus B=\left[\begin{smallmatrix}A&0\cr 0&B\end{smallmatrix}\right]$. $(C_0(A)$ is defined to be the $1\times 1$ identity matrix $[1]$.)
\end{lem}

\begin{proof}
For $T\subset\{1,\dots,n\}$, let $T'=\{m+j:j\in T\}$. Let $C$ be the principal submatrix of $C_{k+l}(A\oplus B)$ labeled by all $S\cup T'$ with $S\in\mathcal C^m_k$ and $T\in\mathcal C^n_l$. For $S_1,S_2\in\mathcal C^m_k$ and $T_1,T_2\in\mathcal C^n_l$, the $(S_1\cup T_1',S_2\cup T_2')$-entry of $C_{k+l}(A\oplus B)$ is 
\[
\det[(A\oplus B)(S_1\cup T_1',S_2\cup T_2')]=\det(A(S_1,S_2))\det(B(T_1,T_2));
\]
see Figure~1.
\begin{figure}
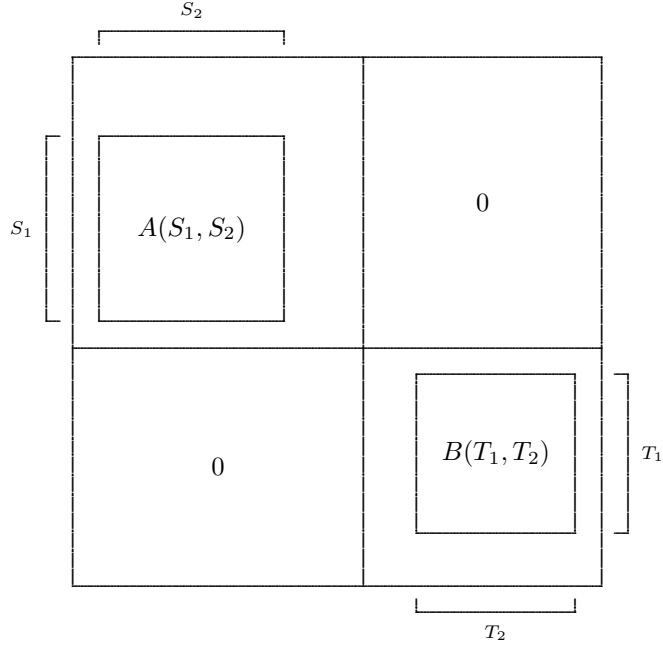
\label{F1} 
\[
\beginpicture
\setcoordinatesystem units <1em,1em> point at 0 0

\plot 0 0  20 0  20 20  0 20  0 0 /
\plot 11 0  11 20 /
\plot 0 9  20 9 /
\plot 1 10  8 10  8 17  1 17  1 10 /
\plot 13 2  19 2  19 8  13 8  13 2 / 
\plot -0.5 10  -1 10  -1 17  -0.5 17 /
\plot 20.5 2  21 2  21 8  20.5 8 /
\plot 1 20.5  1 21  8 21  8 20.5 /
\plot 13 -0.5  13 -1  19 -1  19 -0.5 /
 
\put {$A(S_1,S_2)$}  at 4.5 13.5
\put {$B(T_1,T_2)$}  at 16 5
\put {0} at 5.5 4.5
\put {0} at 15.5 14.5

\put{$\scriptstyle S_1$} [r] at -1.5 13.5
\put{$\scriptstyle S_2$} [b] at 4.5 21.5
\put{$\scriptstyle T_2$} [t] at 16 -1.5
\put{$\scriptstyle T_1$} [l] at 21.5 5
\endpicture
\]
\caption{$A(S_1,S_2)$ and $B(T_1,T_2)$ in $A\oplus B$}
\end{figure}
Hence $C=C_k(A)\otimes C_l(B)$.
\end{proof}

\begin{lem}\label{L3.3}
Let $J_n$ be given by \eqref{2.6.0}. Then
\begin{equation}\label{CrJn}
C_r(J_n)=\left[\begin{matrix} C_r(J_{n-1})&*\cr 0&C_{r-1}(J_{n-1})\end{matrix}\right],
\end{equation}
where the rows and columns of $C_r(J_{n-1})$ are labeled by $\mathcal C^{n-1}_r$ and the rows and columns of $C_{r-1}(J_{n-1})$ are labeled by $\{S\cup\{n\}:S\in\mathcal C^{n-1}_{r-1}\}$.
\end{lem}

\begin{proof}
We have
\[
J_n=\left[\begin{matrix} J_{n-1}&*\cr 0&1\end{matrix}\right].
\]
Let $S,T\in\mathcal C^n_r$. We show that $\det J_n(S,T)$, the $(S,T)$-entry of $C_r(J_n)$, equals the corresponding entry in the right side of \eqref{CrJn}.

If $n\notin S$ and $n\notin T$, then $J_n(S,T)=J_{n-1}(S,T)$, whence $\det J_n(S,T)=\det J_{n-1}(S,T)$.

If $n\in S$ and $n\in T$, write $S=S'\cup\{n\}$ and $T=T'\cup\{n\}$, where $S',T'\in\mathcal C^{n-1}_{r-1}$. Then 
\[
J_n(S,T)=\left[\begin{matrix}J_{n-1}(S',T')&*\cr 0&1\end{matrix}\right],
\]
whence $\det J_n(S,T)=\det J_{n-1}(S',T')$.

If $n\in S$ and $n\notin T$, then the last row of $J_n(S,T)$ is 0, whence $\det J_n(S,T)=0$.
\end{proof}

\begin{lem}\label{L3.4}
We have
\[
\text{\rm rank}(C_r(J_n)-I)\ge\binom{n-1}r,\quad r\ge 1.
\]
\end{lem}

\begin{proof}
Let $\rho(n,r)=\text{rank}(C_r(J_n)-I)$. By Lemma~\ref{L3.3}, we have 
\[
\begin{cases}
\rho(n,r)\ge \rho(n-1,r-1)+\rho(n-1,r),\cr
\rho(n,r)=0\ \text{unless}\ 1\le r\le n-1,\cr
\rho(n,1)=n-1\ \text{for}\ n\ge 1.
\end{cases}
\]
Using these conditions, it follows by induction on $r$ that $\rho(n,r)\ge\binom{n-1}r$ for $r\ge 1$.
\end{proof}

\subsection{Proof of Theorem~\ref{T3.1}}\

We are now ready to prove Theorem~\ref{T3.1}. First, recall that 
\begin{equation}\label{3.2}
|\text{AGL}(n,\f_2)|=(2^n-2^0)(2^n-2^1)\cdots(2^n-2^{n-1})\cdot 2^n=2^{n^2+n}\prod_{i=1}^n(1-2^{-i})
\end{equation}
and 
\begin{equation}\label{3.3}
|R(n-2,n)|=2^{2^n-n-1}.
\end{equation}
For $\sigma\in\text{AGL}(n,\f_2)$, let $\text{Fix}(\alpha)$ be the number of fixed points of $\sigma$ in $R(n-2,n)$. By Burnside's lemma,
\begin{align*}
\frak M_n\,&=\frac 1{|\text{AGL}(n,\f_2)|}\sum_{\sigma\in\text{AGL}(n,\f_2)}\text{Fix}(\sigma)\cr
&=\frac 1{|\text{AGL}(n,\f_2)|}\Bigl(|R(n-2,n)|+\sum_{\text{id}\ne\sigma\in\text{AGL}(n,\f_2)}\text{Fix}(\sigma)\Bigr),
\end{align*}
where
\[
\frac{|R(n-2,n)|}{|\text{AGL}(n,\f_2)|}=\frac{2^{2^n-n^2-2n-1}}{\prod_{i=1}^n(1-2^{-i})}.
\]
Therefore, to prove \eqref{3.1}, it suffices to show that
\begin{equation}\label{3.4}
\sum_{\text{id}\ne\sigma\in\text{AGL}(n,\f_2)}\text{Fix}(\sigma)=o(|R(n-2,n)|)=o(2^{2^n-n-1}).
\end{equation}

Let $\text{id}\ne\sigma=\left[\begin{smallmatrix} A&0\cr a&1\end{smallmatrix}\right]\in\text{AGL}(n,\f_2)$. By \eqref{CA}, the matrix of $\sigma$ with respect to the basis $\{X_S:S\subset\{1,\dots,n\},\ |S|\le n-2\}$ of $R(n-2,n)$ is
\begin{equation}\label{3.5}
\mathcal A=\left[\begin{matrix}
C_{n-2}(A)&0&\cdots&0\cr
*&C_{n-3}(A)&\cdots&0\cr
\vdots &\vdots&\ddots&\vdots\cr
*&*&\cdots&C_0(A)\end{matrix}\right].
\end{equation}
Note that 
\begin{equation}\label{3.6}
\text{Fix}(\sigma)=2^{\text{null}(\mathcal A-I)}.
\end{equation}
We estimate $\text{Fix}(\sigma)$ in several cases.

\medskip
{\bf Case 1.} Assume that $A$ has an eigenvalue $\ne 1$.

By Lemma~\ref{L3.1}, the algebraic multiplicity of the eigenvalue $1$ of $\mathcal A$ is $\le \frac 12\sum_{r=0}^{n-2}\binom nr<2^{n-1}$. Hence
\begin{equation}\label{3.7}
\text{Fix}(\sigma)<2^{2^{n-1}}.
\end{equation}

\medskip
{\bf Case 2.} Assume that $1$ is the only eigenvalue of $A$ and $A$ has an elementary divisor $(X-1)^m$ with $m\ge \lfloor n/2\rfloor+1$.

We may assume that $A=J_m\oplus A_1$ for some $A_1\in\text{GL}(n-m,\f_2)$. By Lemma~\ref{L3.2}, $C_r(J_m)$ is a principal matrix of $C_r(A)$ for all $0\le r \le m$. Thus by \eqref{3.5} and Lemma~\ref{L3.4},
\begin{align*}
\text{rank}(\mathcal A-I)\,&\ge\sum_{r=0}^{n-2}\text{rank}(C_r(A)-I)\ge\text{rank}(C_3(A)-I)\cr
&\ge\text{rank}(C_3(J_m)-I)\ge\binom{m-1}3\ge\binom{\lfloor n/2\rfloor}3.
\end{align*}
Hence
\begin{equation}\label{3.8}
\text{Fix}(\sigma)\le 2^{\binom n0+\cdots+\binom n{n-2}-\binom{\lfloor n/2\rfloor}3}<2^{2^n-\binom{\lfloor n/2\rfloor}3}.
\end{equation}

\medskip
{\bf Case 3.} Assume that $1$ is the only eigenvalue of $A$ and $A$ has an elementary divisor $(X-1)^m$ with $2\le m\le \lfloor n/2\rfloor$. 

Again, we may assume that $A=J_m\oplus A_1$ for some $A_1\in\text{GL}(n-m,\f_2)$. By Lemma~\ref{L3.2}, $C_1(J_m)\otimes C_3(A_1)$ is a principal submatrix of $C_4(A)$. Since 
\[
C_1(J_m)\otimes C_3(A_1)=\left[\begin{matrix}
C_3(A_1)&\kern-0.8em C_3(A_1)\cr
&\cdot&\cdot\cr
&&\cdot&\kern0.8em \cdot\cr
&&&\kern0.8em \cdot&C_3(A_1)\cr
&&&&C_3(A_1)\end{matrix}\right]_{m\times m\ \text{blocks}}
\]
and $m\ge 2$, we have
\[
\text{rank}(C_1(J_m)\otimes C_3(A_1)-I)\ge\text{rank}\, C_3(A_1)=\binom{n-m}3\ge\binom{n/2}3.
\]
Therefore,
\[
\text{rank}(\mathcal A-I)\ge\text{rank}(C_4(A)-I)\ge \text{rank}(C_1(J_m)\otimes C_3(A_1)-I)\ge \binom{n/2}3,
\]
and hence
\begin{equation}\label{3.9}
\text{Fix}(\sigma)<2^{2^n-\binom{n/2}3}.
\end{equation}

\medskip
{\bf Case 4.} Assume that $A=I$ but $a\ne(0,\dots,0)$.

We may assume that $a=(0,\dots,0,1)$, i.e.,
\[
\sigma(x)=x+(0,\dots,0,1)\quad\text{for all}\ x\in\f_2^n.
\]
In this case, for $f\in R(n-2,n)$,
\[
\sigma(f)=f \Leftrightarrow f=f(X_1,\dots,X_{n-1})\in R(n-2,n-1).
\]
Thus
\begin{equation}\label{3.10}
\text{Fix}(\sigma)=|R(n-2,n-1)|=2^{2^{n-1}-1}.
\end{equation}

\medskip
In all four cases, we always have
\[
\text{Fix}(\sigma)<2^{2^n-\binom{\lfloor n/2\rfloor}3}\quad \text{for $n$ sufficiently large}.
\]
Therefore
\[
\sum_{\text{id}\ne\sigma\in\text{AGL}(n,\f_2)}\text{Fix}(\sigma)<|\text{AGL}(n,\f_2)|2^{2^n-\binom{\lfloor n/2\rfloor}3}\le 2^{n^2+n+2^n-\binom{\lfloor n/2\rfloor}3}=o(2^{2^n-n-1}).
\]
This completes the proof of Theorem~\ref{T3.1}.

\section{Final Remarks}

In general, let $\theta(n;s,t)$ denote the number of AGL orbits of $R(r,n)/R(s-1,n)$, $0\le s\le r\le n$. (Thus $\frak N_{2,n}=\theta(n;0,n)$ and $\frak M_n=\theta(n;0,n-2)$.) Since $\theta(n;s,r)=\theta(n;n-r,n-s)$ \cite[Theorem~5.1]{Hou-JA-1995}, we may assume that $s+r\le n$, that is, $0\le s\le r\le n-s$. In this range, $\theta(n;s,r)$ is numerically computed for $n\le 10$ \cite{Zeng-Yang-arXiv1912.11189} and is theoretically determined for $r\le 2$ (linear and quadratic functions) and for $(s,r)=(0,n)$ (this paper). It appears that with due effort, $\theta(n;0,n-1)$ can also be determined theoretically. For other values of $(s,r)$, explicit formulas for $\theta(n;s,r)$ appears to be out of immediate reach. For example, to determine $\theta(n;3,3)$, one needs to know $\text{null}(C_3(A)-I)$ for every $A\in\text{GL}(n,\f_2)$ in a canonical form under conjugation, or one needs to know the classification of cubic forms over $\f_2$; the former is difficult and the latter is probably impossible.

As for the asymptotics, we have only solved the question for $\theta(n;0,n-2)$. However, it seems that the method should work for all $(s,r)$.

\section{Conclusion}

We derived an explicit formula for the number of equivalence classes of functions from $\f_q^n$ to $\f_q$ under the action of the affine linear group $\text{AGL}(n,\f_q)$. These numbers are enormous unless both $q$ and $n$ are small, hence complete classification of functions from $\f_q^n$ to $\f_q$ is not practical. However, the group theoretic approach in the paper may lead to solutions of similar problems. We also proved an asymptotic formula for the number of equivalence classes of cosets of the first order Reed-Muller code under the action of  $\text{AGL}(n,\f_2)$. The asymptotic formula indicates that for most cosets of the first order Reed-Muller code, the subgroup of  $\text{AGL}(n,\f_2)$ that stabilizes them is trivial.

\section*{Appendix}

\centerline{List of Notation}

\medskip

\begin{tabbing}
\= $\oplus $ \kern3.1cm \= $A\oplus B=\left[\begin{smallmatrix}A\cr &B\end{smallmatrix}\right]$ \\
\> $\boxplus $ \> $\left[\begin{smallmatrix}A\cr a&1\end{smallmatrix}\right]\boxplus \left[\begin{smallmatrix}B\cr b&1\end{smallmatrix}\right]=\left[\begin{smallmatrix}A\cr &B\cr  a&b&1\end{smallmatrix}\right]$ \\

\> \> \\

\> $\mathcal A $ \> \eqref{3.5} \\
\> $a_1(\lambda,(\boldsymbol{\lambda}_d)_{d\in D}) $ \> $|c(\alpha)|$, formula in \eqref{2.13} \\
\> $a_2(t,\lambda,(\boldsymbol{\lambda}_d)_{d\in D}) $ \> $|c(\beta)|$, formula in \eqref{2.14} \\
\> $A(S,T) $ \> submatrix of $A$ with row (column) indices in $S$ ($T$) \\
\> $\text{AGL}(n,\f_q)$ \>  $\{\left[\begin{smallmatrix}A&0\cr a&1\end{smallmatrix}\right]:A\in\text{GL}(n,\f_q),\ a\in\f_q^n\}$, affine linear group \\
\> $b_1(\lambda,(\boldsymbol{\lambda}_d)_{d\in D}) $ \> $o(\alpha)$, formula in \eqref{2.20} \\
\> $b_2(t,\lambda,(\boldsymbol{\lambda}_d)_{d\in D}) $ \> $o(\beta)$, formula in \eqref{2.21} \\
\> $\mathcal C $ \> set of representatives of conjugacy classes of $\text{AGL}(n,\f_q)$ \\
\> $\mathcal C_1,\mathcal C_2 $ \> defined in \eqref{C1} and \eqref{C2} \\
\> $\mathcal C^n_r $ \> set of subsets of $\{1,\dots,n\}$ of size $r$ \\
\> $C_r(A) $ \> $r$th compound matrix of $A$ \\
\> $c(\alpha) $ \> centralizer of $\alpha$ in $\text{AGL}(n,\f_q)$ \\
\> $D $ \> $\{d>1:d \mid q^i-1\ \text{for some}\ 1\le i\le n\}$ \\
\> $e_1(\lambda,(\boldsymbol{\lambda}_d)_{d\in D}) $ \> defined by $\text{Fix}(\alpha)=q^{e_1(\lambda,(\boldsymbol{\lambda}_d)_{d\in D})}$, formula in \eqref{2.22} \\
\> $e_2(t,\lambda,(\boldsymbol{\lambda}_d)_{d\in D}) $ \> defined by $\text{Fix}(\beta)=q^{e_2(t,\lambda,(\boldsymbol{\lambda}_d)_{d\in D})}$, formula in \eqref{2.23} \\
\> $f^{\lambda} $ \>  $\{\underbrace{f^1,\dots,f^1}_{\lambda_1},\underbrace{f^2,\dots,f^2}_{\lambda_2},\dots\}$, where $\lambda=(\lambda_1,\lambda_2,\dots)$ \\
\> $\mathcal F(\f_q^n,\f_q)$  \>  set of functions from $\f_q^n$ to $\f_q$\\
\> $\text{Fix}(\alpha)$ (\S3) \> number of fixed points of $\alpha$ in $\mathcal F(\f_q^n,\f_q)$ \\
\> $\text{Fix}(\sigma)$ (\S4) \> number of fixed points of $\sigma$ in $R(n-2,n)$ \\
\> $\text{fix}(\alpha)$ \> number of fixed points of $\alpha$ in $\f_q^n$ \\
\> $I$ \> identity matrix \\
\> $\mathcal I $ \> set of monic irreducible polynomials in $\f_q[X]\setminus\{X\}$ \\
\> $I_d $ \> $\{f\in\mathcal I:\text{ord}\, f=d\}$ \\
\> $\text{id} $ \> identity of the affine linear group \\
\> $\frak M_n $ \> number of AGL orbits of $R(n,n)/R(1,n)$ \\
\> $\frak m(\lambda) $ \> $\max\{i:\lambda_i>0\}$, where $\lambda=(\lambda_1,\lambda_2,\dots)\in\mathcal P$ \\
\> $\frak m(\boldsymbol{\lambda}) $ \> $\max_{1\le i\le\psi(d)}\frak m(\boldsymbol{\lambda}^{(i)})$, where $\boldsymbol{\lambda}=(\boldsymbol{\lambda}^{(1)},\dots,\boldsymbol{\lambda}^{(\psi(d))})\in\Lambda_d$ \\
\> $\frak N_{q,n} $ \> number of AGL orbits of $\mathcal F(\f_q^n,\f_q)$ \\
\> $\text{null}(A) $ \> nullity of $A$ \\
\> $o(\ ) $  (\S3)\> order of a group element \\
\> $o(\ ) $ (\S4) \> little-$o$ asymptotic \\
\> $o_d(q) $ \> multiplicative order of $q$ in $\Bbb Z/d\Bbb Z$ \\
\> $\text{ord}\, f $ \> order of $f\in\mathcal I$ \\
\> $\mathcal P $ \> set of all partitions \\
\> $R_q(r,n)$ \>  $\{f\in\mathcal F(\f_q^n,\f_q):\deg f\le r\}$, $q$-ary Reed-Muller code\\
\> $R(r,n)$ \>  $R_2(r,n)$, binary Reed-Muller code\\
\> $s(\boldsymbol{\lambda}) $ \> number of permutations of $\boldsymbol{\lambda}=(\boldsymbol{\lambda}^{(1)},\dots,\boldsymbol{\lambda}^{(\psi(d))})\in\Lambda_d$ \\
\> $T(\lambda) $ \> $\{i:\lambda_i>0\}$, where $\lambda=(\lambda_1,\lambda_2,\dots)\in\mathcal P$ \\
\> $X_S $ \> $\prod_{i\in S}X_i$ \\

\> \> \\

\> $ \alpha, \beta$  \>  defined in \eqref{2.10} and \eqref{2.12} and protected in \S3\\
\> $\epsilon_t $ \> $(1,0,\dots,0)\in\f_q^t$ \\
\> $\epsilon(d,k) $ \> \eqref{2.18} \\
\> $\theta(n;s,r) $ \> number of AGL orbits of $R(r,n)/R(s-1,n)$  \\
\> $|\lambda| $ \> $\sum_{i\ge 1}i\lambda$, where $\lambda=(\lambda_1,\lambda_2,\dots)\in\mathcal P$ \\
\> $|\boldsymbol{\lambda}| $ \> $\sum_{i=1}^{\psi(d)}|\boldsymbol{\lambda}^{(i)}|$, where $\boldsymbol{\lambda}=(\boldsymbol{\lambda}^{(1)},\dots,\boldsymbol{\lambda}^{(\psi(d))})\in\Lambda_d$  \\
\> $\Lambda_d $ \> $\{(\boldsymbol{\lambda}^{(1)},\dots,\boldsymbol{\lambda}^{(\psi(d))}):\boldsymbol{\lambda}^{(i)}\in\mathcal P,\ \boldsymbol{\lambda}^{(1)}\le\cdots\le\boldsymbol{\lambda}^{(\psi(d))}\}$, $d\in D$  \\
\> $\nu(\ )$  \> $p$-adic order \\
\> $\sigma_{f^\lambda} $ \> matrix with elementary divisors $f^\lambda$ \\
\> $\sigma_\lambda $ \> \eqref{sig-lam} \\
\> $\sigma_{\lambda,t} $ \> \eqref{sig-lam-t} \\
\> $\phi $ \> Euler totient function \\
\> $\psi(d) $ \>  $\phi(d)/o_d(q)$ \\
\> $\Omega $ \> $\{(\lambda,(\boldsymbol{\lambda}_d)_{d\in D}):\lambda\in\mathcal P,\, \boldsymbol{\lambda}_d\in\Lambda_d,\,|\lambda|+\!\sum_{d\in D}\!o_d(q)|\boldsymbol{\lambda}_d|=n\}$ \\
\end{tabbing}




\begin{thebibliography}{99}

\bibitem{Berlekamp-Welch-IEEE-IT-1972}
E. R. Berlekamp and L. R. Welch, {\it Weight distributions of the cosets of the (32,6) Reed-Muller code}, 
IEEE Trans. Inform. Theory {\bf 18} (1972), 203 -- 207.

\bibitem{Budaghyan-Carlet-Pott-IEEE-IT-2006}
L. Budaghyan, C. Carlet, A. Pott, {\it New classes of almost bent and almost perfect nonlinear polynomials}, IEEE Trans. Inform. Theory {\bf 52} (2006),  1141 -- 1152.

\bibitem{Harrison-JSIAM-1964}
M. A. Harrison, 
{\it On the classification of boolean functions by the general linear and affine groups}, J. Soc. Indust. Appl. Math. {\bf 12} (1964), 285 -- 299.

\bibitem{Hou-JA-1995} 
X. Hou {\it $\text{\rm AGL}(m,2)$ acting on $R(r,m)/R(s,m)$}, J. Algebra
{\bf 171} (1995), 921 -- 938.

\bibitem{Hou-DM-1996}
X. Hou, {\it $\text{\rm GL}(m,2)$ acting on $R(r,m)/R(r-1,m)$},
Discrete Math. {\bf 149} (1996), 99 -- 122.

\bibitem{Hou-ams-gsm-2018}
X. Hou, {\it Lectures on Finite Fields}, Graduate Studies in Mathematics 190, American Mathematical Society, Providence, RI, 2018.


\bibitem{MacWilliams-Sloane-1977}
F. J. MacWilliams and N. J. A. Sloane, {\it The Theory of Error-correcting Codes, II},  North-Holland Publishing Co., Amsterdam -- New York -- Oxford, 1977.

\bibitem{Maiorana-MC-1991}
J. A. Maiorana,  {\it A classification of the cosets of the Reed-Muller code $\mathcal R(1,6)$}, Math. Comp. {\bf 57} (1991), 403 -- 414.

\bibitem{Nechiporuk-DKNSSSR-1958}
E. I. Nechiporuk, {\it On the synthesis of networks using linear transformations of variables}, Dokl. Akad. Nauk. SSSR {\bf 123} (1958), 610 -- 612. Available in English in Automation Express, April 1959, 12 -- 13.

\bibitem{oeis} 
The On-Line Encyclopedia of Integer Sequences,
https:/\!/oeis.org/

\bibitem{Strazdins-AAM-1997}
I. Strazdins, {\it Universal affine classification of Boolean functions}, Acta Appl. Math. {\bf 46} (1997), 147 -- 167.

\bibitem{Wedderburn-1964}
J. H. M. Wedderburn,  {\it Lectures on Matrices}, Dover Publications, Inc., New York, 1964.

\bibitem{Zeng-Yang-arXiv1912.11189}
X. Zeng and G. Yang, {\it Computing the number of affine equivalence classes of Boolean functions modulo functions of different degrees}, arXiv:1912.11189.

\bibitem{Zhang-Yang-Hung-Zhang-IEEE-C-2016}
Y. Zhang, G. Yang, W. N. N. Hung, J. Zhang, {\it Computing affine equivalence classes of Boolean functions by group isomorphism}, IEEE Trans. Comput. {\bf 65} (2016), 3606 -- 3616.

\bibitem{Zivkovic-Caric-IEEE-IT-2021}
M. \v Zivkovi\'c and M. Cari\'c, {\it On the number of equivalence classes of boolean and invertible boolean functions}, IEEE Trans. Inform. Theory {\bf 67} (2021), 391 -- 407.


\end{thebibliography}
\end{document}